\documentclass[11pt]{amsart}

\usepackage[T1]{fontenc}
\usepackage{lmodern}
\usepackage{amsmath,amssymb,amsthm,mathtools}
\usepackage[margin=1in]{geometry}
\usepackage{microtype}
\usepackage{hyperref}
\usepackage{xcolor}
\usepackage{esint}
\usepackage[
    backend=bibtex,
    style=numeric,
    bibencoding=utf8,
    language=auto,
    autolang=other,
    maxnames=10
]{biblatex}
\newtheorem{theorem}{Theorem}
\newtheorem{lemma}[theorem]{Lemma}
\theoremstyle{remark}
\newtheorem{remark}[theorem]{Remark}
\theoremstyle{remark}
\newtheorem*{ack}{Acknowledgments}
\newtheorem*{aidisc}{AI use disclosure}

\newcommand{\HH}{\mathbb H}

\title{Stable Constant Mean Curvature Hypersurfaces in $\HH^n$}

\date{September 27, 2026}
\author{Filippo Gaia}
\address{Department of Mathematics\\Stanford University\\Stanford, CA 94305\\USA}
\email{fgaia@stanford.edu}
\author{Rafe Mazzeo}
\address{Department of Mathematics\\Stanford University\\Stanford, CA 94305\\USA}
\email{rmazzeo@stanford.edu}
\author{Ivan Miranda}
\address{IMPA– Instituto de Matemática Pura e Aplicada\\ Rio de Janeiro, RJ\\ Brasil, 22460-320.}
\email{ivan.miranda@impa.br}
\begin{document}
\begin{abstract}
    For every \(n\geq 4\) and every \(H\) in a neighborhood of \(n-1\) (depending on $n$), we prove the existence of a properly embedded strongly stable CMC hypersurface in \(\mathbb H^n\) with mean curvature \(H\) and infinitely many ends, invariant under the action of a Schottky group. In particular, stable Bernstein-type rigidity fails in this range.
\end{abstract}
\maketitle
\section{Introduction}
In this paper we prove the following result.

\begin{theorem}\label{thm:main}
For each integer $n\geq 4$, there exists $\delta>0$ such that for any $H \in(n-1-\delta,n-1+\delta)$, there exists a properly embedded, noncompact hypersurface $\Sigma_H^{n-1}\subset\HH^n$
with constant mean curvature $H$ which is strictly strongly stable, has infinitely many ends, and is invariant under the action of a Schottky group.
\end{theorem}
\textit{Strongly stable} means that 
\begin{align}
\label{eq:stability-definition}
Q_{\Sigma_H}^{\HH^n}(f):=\int_{\Sigma_H}  \bigl(|\nabla f|^2-(|A|^2-(n-1))f^2\bigr)\,d\mu\geq0\ \ \mbox{for every} \ \ f\in C_c^\infty(\Sigma_H), \ f \not\equiv 0;
\end{align}
\textit{{Strictly strongly stable}} means that $Q_{\Sigma_H}^{\HH^n}(f)\geq \eta_{H,n}\int_{\Sigma_H}f^2\,d\mu$ for any non-zero $f$, for some $\eta_{H,n}>0$. By contrast, \textit{weakly stable} means that \eqref{eq:stability-definition} holds for 
any $f\in C_c^\infty(\Sigma_H)$ with $\int_{\Sigma_H} f\,d\mu=0.$\\ 

\medskip

We briefly explain the motivation for studying this problem. In recent years, substantial progress has been made on the following question, usually referred to as the \emph{stable Bernstein problem}:
\[
\textit{Must every complete, connected, two-sided, stable minimal hypersurface \(\Sigma\) in \(\mathbb R^n\) be a hyperplane?}
\]
In this setting, stability means that
\[
Q_{\Sigma}^{\mathbb R^n}(f)
:=
\int_\Sigma
\left(
|\nabla f|^2-|A|^2f^2
\right)d\mu
\geq 0
\qquad
\text{for every }f\in C_c^\infty(\Sigma).
\]

A positive answer for \(n=3\) was given independently by do Carmo--Peng, Fischer-Colbrie--Schoen, and Pogorelov
\cite{doCarmoPeng1979,FischerColbrieSchoen1980,Pogorelov1981}.
The case \(n=4\) was proved by Chodosh and Li
\cite{ChodoshLi2024}; see also their alternative proof
\cite{ChodoshLi2023} and the independent proof of
Catino--Mastrolia--Roncoroni \cite{CatinoMastroliaRoncoroni2024}.
The cases \(n=5\) and \(n=6\) were subsequently settled by
Chodosh--Li--Minter--Stryker \cite{ChodoshLiMinterStryker2026}
and Mazet \cite{Mazet2024}, respectively. Finally, a recent
preprint of Hong--Li--Wang proves the result for
\(n=7\) \cite{HongLiWang2026}. By contrast, the answer is negative
in every ambient dimension \(n\geq8\), as there exist smooth area-minimizing hypersurfaces asymptotic to the Simons cone
\cite{Simons1968,BombieriDeGiorgiGiusti1969,HardtSimon1985}.

Similar rigidity problems have been studied in other space forms. In particular, in the hyperbolic space $\HH^n$ the stability condition for 
constant mean curvature (CMC) $H = n-1$ takes the form 
\begin{align*} Q_\Sigma^{\mathbb H^n}(f)=\int_\Sigma (\lvert \nabla f\rvert^2-\lvert\mathring{A}\rvert^2f^2)\,d\mu, \end{align*}
where $\mathring{A}$ denotes the trace-free part of $A$.
In analogy to the stable Bernstein problem, it is natural to ask the following question:
\begin{equation}\label{q: question-1}
\parbox{0.85\textwidth}{\centering\itshape
Must every complete, connected, two-sided, strongly stable CMC
hypersurface $\Sigma\subset\mathbb H^n$ with $H=n-1$ satisfy
$\mathring A=0$?
}
\end{equation}
Such a hypersurface would be totally umbilic and hence a horosphere.
In ambient dimension \(n=3\), Silveira answered this question
affirmatively \cite{daSilveira1987}. In higher ambient dimensions,
the question remained open until recently (see Remark \ref{rem: zwang}). For larger values of $H$, rigidity results are known also in higher dimension: 
Q. Deng showed in \cite{Deng2011} that if $\Sigma$ is weakly stable, if $n=4$ and $H>\sqrt{\frac{64}{63}}(n-1)$ (or if $n=5$ and 
$H>\sqrt{\frac{175}{148}}(n-1)$), then $\Sigma$ must be compact. This improved previous results by X. Cheng \cite{XuCheng}. Moreover, 
for $n=4$ and $H>n-1$, H. Hong showed in \cite{Hong2025} that if $\Sigma$ is weakly stable (or has finite weak index), has finitely many ends 
and $\dim H_c^1(\Sigma)<\infty$, then $\Sigma$ must be compact. In his work H. Hong asks whether this rigidity result remains valid without 
assumption on the topology or the number of ends of $\Sigma$. Similarly, Chodosh asked in \cite{Chodosh2026} whether a two-sided stable 
CMC immersion with mean curvature $H\geq n-1$  must be a horosphere (for $H=n-1$, this question is equivalent to \eqref{q: question-1}).
When $n=6$, the third-named author proved a partial positive answer to Chodosh's question \cite{miranda2026} building on the works of Mazet \cite{Mazet2024} and Chen-Hong-Li \cite{CHL}. 

Theorem \ref{thm:main} shows that for any $n\geq 4$, the answer to \eqref{q: question-1} and to the questions of H. Hong and Chodosh is negative. 

\begin{remark}
Note that every non-minimal CMC hypersurface is two-sided, and its orientation can be chosen so that $H>0$. In addition, by the Hopf-Rinow theorem,
every properly embedded hypersurface in $\HH^n$ is complete. 
\end{remark}
\begin{remark}\label{rem: zwang}
During the preparation of this work, we became aware of a new paper by  Z. Wang \cite{wang2026stronglystablecmconehypersurfaces}
who constructs simpler examples of complete CMC hypersurfaces in $\HH^n$ with $H=n-1$ for any $n\geq 4$. His examples have finite topology 
and two ends. However \cite{Hong2025} shows that there are no such examples in $\HH^4$ which have $H>3$, while our results cover this case as well.

We were also informed that in another recent paper, Wang \cite{wang2026finiteindexconstantmean} gives a partial positive answer to 
Chodosh's question when $n=7$ and improves the threshold for rigidity for $n \in \{4,5,6\}$.
\end{remark}

\subsection{Idea of the proof}
We begin by constructing CMC hypersurfaces in a convex cocompact hyperbolic quotient $X=\HH^n/\Gamma$ which have constant mean curvature $H$ 
ranging in an interval around $n-1$. To this end we follow an idea of the second-named author and Pacard \cite{MazzeoPacard2011}, where CMC foliations of
asymptotically hyperbolic manifold near infinity are constructed by perturbing level sets of the renormalized distance from the ideal boundary. The 
natural compactification of $X$ comes with a natural conformal class $\mathfrak c$ on its  boundary at infinity $Y$. By Fefferman-Graham \cite{FeffermanGraham2012}, 
any choice of a metric $h \in \mathfrak c$ uniquely determines a geodesic defining function $x_h$ for this ideal boundary. The local-at-infinity CMC
foliations are in bijective correspondence with choices of metrics $h$ in this conformal class with constant scalar curvature $R_h$.  Each level set 
$x_h = \varepsilon$ is then perturbed slightly to have constant mean curvature equal to
\begin{equation}\label{eq: H(e,t)}
H=n-1+\frac{R_h}{2(n-2)}\varepsilon^2+O(\varepsilon^3).
\end{equation}
Now, by a result of Nayatani \cite{Nayatani1997}, the Yamabe constant of $(Y,\mathfrak c)$ has the same sign as 
\begin{align*}
\frac{n-3}{2}-\delta(\Gamma),
\end{align*}
where $\delta(\Gamma)$ is the \textit{critical exponent} of $\Gamma$. By \cite{Sullivan1984}, $\delta(\Gamma) = \dim_{\mathcal H}\Lambda_\Gamma$,
where $\Lambda_\Gamma$ is the \textit{limit set} of $\Gamma$. We define these carefully  below. 

We then construct a family of groups $\Gamma_T$ depending analytically on the parameter $T$, which have critical exponent ranging in a 
neighborhood of $\frac{n-3}{2}$. 
Identifying these quotients in a suitable sense, we thus obtain the smooth family of metrics $h_T$ with constant scalar curvature $R_T$ ranging in a 
neighborhood of $0$. The CMC hypersurface $\Sigma_{\varepsilon, T}$ has mean curvature $H_{\varepsilon, T}$ given by \eqref{eq: H(e,t)} and depends continuously on $\varepsilon$ and $T$. Therefore, as $\varepsilon$ and $T$ vary,  $H_{\varepsilon, T}$ covers the desired range.\\

\smallskip

Next we show that the lift $\tilde \Sigma_{\varepsilon, T}$ of $\Sigma_{\varepsilon, T}$ to $\mathbb H^n$ satisfies the desired properties. We focus here on strict 
strong stability. To this end we observe that the groups $\Gamma_T$ can be chosen such that $\Gamma_T\simeq F_N$ is the free group with $N$ generators,
$N \geq 2$.  For these groups, the inclusion map $\iota: Y\to \overline{X}$ induces an isomorphism
\[
\iota_\ast: \pi_1(Y)\to \pi_1(\overline{X})\simeq \pi_1(X).
\]
This implies in particular that $\pi_1(\Sigma_{\varepsilon, T})\simeq F_N$ and that $\tilde \Sigma_{\varepsilon, T}$ is the universal cover of $\Sigma_{\varepsilon, T}$. As $F_N$ is not amenable, a result of Brooks \cite{Brooks1981} implies that the bottom of the spectrum of the Laplacian is positive for $\tilde \Sigma_{\varepsilon, T}$ with the metric $g_{\tilde \Sigma_{\varepsilon, T}}$ induced by $\HH^n$. In order to get a uniform lower bound for this quantity, we use the fact that by the construction of $\Sigma_{\varepsilon, T}$ described above, the metric $g_{\tilde \Sigma_{\varepsilon, T}}$ has the form $g_{\tilde \Sigma_{\varepsilon, T}}=\varepsilon^{-2} \widetilde{h_T}+O(1)$, where $\widetilde{h_T}$ is the pull-back of $h_T$ by the quotient map (this asymptotic holds after identifying $\tilde \Sigma_{\varepsilon, T}$ with the universal covering $\tilde Y$ of $Y$). This yields
\begin{align}
    \label{eq: lower-est-lambda-0}
    \lambda_0(\tilde \Sigma_{\varepsilon, T}, g_{\tilde \Sigma_{\varepsilon, T}})\geq c\varepsilon^2
\end{align}
for some positive constant $c$, uniformly in $T\in I$ (after shrinking $I$ if necessary).
The construction of $\Sigma_{\varepsilon, T}$ as a graph over $Y$ also implies that the shape operator of $\tilde \Sigma_{\varepsilon, T}$ satisfies
\begin{align}
    \label{eq: est-shape-op-intro}
    \lvert A_{\tilde\Sigma_{\varepsilon, T}}\rvert^2-(n-1)=\varepsilon^2O(\lvert T-T_0\rvert)+o(\varepsilon^2).
\end{align}
Combining \eqref{eq: lower-est-lambda-0} and \eqref{eq: est-shape-op-intro} we show that for any $f\in C_c^\infty(\tilde \Sigma_{\varepsilon, T})$
\[
Q^{\HH^n}_{\tilde \Sigma_{\varepsilon, T}}(f)\geq \varepsilon^2\left(c+O(\lvert T-T_0\rvert)+o_\varepsilon(1)\right)\int_{\tilde \Sigma_{\varepsilon, T}}f^2d\mu_{g_{\tilde\Sigma_{\varepsilon, T}}}.
\]
By taking $I$ and $\varepsilon_0$ smaller if necessary, we can ensure that if $f$ is not zero, $Q^{\HH^n}_{\tilde \Sigma_{\varepsilon, T}}(f)>0$.\\
Finally, note that $\widetilde\Sigma_{\varepsilon, T}$ is obtained by gluing copies of a
fundamental region of $\Sigma_{\varepsilon, T}$ according to the Cayley graph of
$\pi_1(\Sigma)\simeq F_N$, therefore it has infinitely many ends. Since a complete, connected, totally umbilic
hypersurface in $\HH^n$ has at most one end, $\widetilde\Sigma_{\varepsilon,T}$
cannot be umbilic.

\begin{remark}
    As required by \cite{daSilveira1987}, this construction cannot be repeated for $n=3$. In fact, in this case the Schottky group
construction of Theorem \ref{thm:schottky-groups} produces a quotient
\(X_T\) whose ideal boundary \(Y_T\) is a closed surface of genus \(N\),
where \(N\geq2\) is the number of generators. Hence $\chi(Y_T)=2-2N<0$.
The Gauss-Bonnet theorem therefore implies that \(Y_T\) supports no
metric of constant nonnegative scalar curvature. Moreover, when
\(n=3\), the natural map
\[
    \pi_1(Y_T)\longrightarrow\pi_1(X_T)
\]
is no longer an isomorphism.
\end{remark}

\subsection{Related literature}

The problem of classifying weakly stable CMC hypersurfaces immersed in special classes of Riemannian manifolds has several applications and there is a rich literature on this topic. For instance, this study is directly connected to the isoperimetric problem in the compact setting (see \cite{ros2005isoperimetric}). When allowing non-compact examples, these classification results are used to derive curvature estimates for stable CMC hypersurfaces (\cite{Rosenberg-Souam-Toubiana}, \cite{Chodosh-Li-Stryker-4mfds}, \cite{miranda2026}), to develop the min-max theory for the existence of CMC hypersurfaces in generic ambient manifolds \cite{mazurowski2024infinitely}, to derive the so-called maximum principle at infinity in the sense of Ros and Rosenberg \cite{Ros-Rosenberg}, or to study foliations and laminations by CMC hypersurfaces. For an overview of related results see \cite{meeks2016constant,SurveyNelli}. Foliations of asymptotically hyperbolic (or Euclidean) spaces by CMC leaves have attracted special interest because of their connection with mathematical physics and related research areas (see, for instance, \cite{foliation-cmc-asympt-flat},\cite{MazzeoPacard2011}). We refer the reader to the introduction of \cite{Barbosa-Manfredo} for a geometric and variational interpretation of the notions of weak stability and strong stability for CMC hypersurfaces.

With recent advances in the stability analysis of CMC hypersurfaces, after the breakthrough of Chodosh and Li \cite{ChodoshLi2024}, taking into account several contributions, we now have a good understanding of the classification of complete weakly stable CMC hypersurfaces immersed in simply connected space forms of non-negative sectional curvature up to dimension six. In Euclidean spaces of low dimension, the complete, connected, immersed two-sided weakly stable CMC hypersurfaces are the round spheres and the affine hyperplanes (for the minimal case, see the beginning of the introduction and \cite{BELLETTINI2019133}, for the non-minimal case see \cite{Barbosa-Manfredo,Elbert-Nelli-Rosenberg,XuCheng,miranda2026,CHL}); see also \cite{HongLiWang2026,wang2026finiteindexconstantmean} for recent preprints in dimension seven. In the round spheres of low dimension, the complete immersed two-sided weakly stable CMC hypersurfaces are the geodesic spheres (\cite{CatinoMastroliaRoncoroni2024,miranda2026,shen-ye,FischerColbrieSchoen1980,Barbosa-Manfredo-Eschenburg,HongYan2025}); see also the recent preprint in dimension seven \cite{wang2026finiteindexconstantmean}. However, it is perhaps surprising that we still lack a good understanding of non-compact stable CMC hypersurfaces immersed in the hyperbolic spaces of low dimension. This is in contrast with the classical result of Silveira \cite{daSilveira1987} which, as mentioned before, gives a complete picture in dimension three: a complete, noncompact, weakly stable CMC hypersurface immersed in the hyperbolic space with mean curvature $H\ge2$ is necessarily a horosphere. We note that horospheres are strongly stable, but were classified under the weak stability hypothesis. It should be remarked that strongly stable CMC examples with mean curvature $H<2$ abound, as was noticed by Silveira \cite{daSilveira1987}.  

Chodosh and Li connected the stability problem for minimal hypersurfaces with the theory of positive scalar curvature in their seminal work \cite{ChodoshLi2023}. This connection was explored further in the minimal case in \cite{ChodoshLiMinterStryker2026},  \cite{Mazet2024} and \cite{Chodosh-Li-Stryker-4mfds}, and these ideas were developed in the CMC case in \cite{Hong2025},  \cite{HongYan2025} and \cite{miranda2026}. One of the heuristics in these problems is that the stability inequality provides a spectral control on scalar curvature, and three-manifolds with positive scalar curvature in spectral sense share macroscopic one-dimensionality properties (see the introduction of \cite{ChodoshLiMinterStryker2026} and references therein). Under additional topological hypotheses on the manifold, such as finite topology, one can leverage the aforementioned geometric properties to control the volume growth of the hypersurface. 

Volume growth conditions are closely related to stability problems for CMC hypersurfaces (see, for instance, the seminal work of Schoen, Simon, and Yau \cite{SSY}). We highlight here a result by Ilias, Nelli, and Soret \cite{IliasNelliSoret2016} in the context of the hyperbolic space: a complete strongly stable CMC hypersurface immersed in the hyperbolic space $\mathbb{H}^n$ with mean curvature $H>n-1$ must have exponentially fast volume growth.

As a consequence of the so-called Schoen-Yau rearrangement of the stability inequality, every strongly stable CMC hypersurface immersed in $\mathbb{H}^n$ with mean curvature $H>n-1$ has positive scalar curvature in spectral sense. These heuristics are directly related to Hong's theorem \cite{Hong2025}, that under finite topology assumptions, a complete strongly stable CMC hypersurface with mean curvature $H>3$ immersed in $\mathbb{H}^4$ is necessarily compact (see the alternative proof in the appendix of \cite{miranda2026}).

Therefore, a negative answer to Hong's and Chodosh's question in the range $H>n-1$ must be given by a complete CMC hypersurface with exponentially fast volume growth, and for $n=4$ it cannot have finite topology. On the other hand, it must have bounded second fundamental form in low dimensions (\cite{Rosenberg-Souam-Toubiana},\cite{Chodosh-Li-Stryker-4mfds},\cite{miranda2026}) and positive scalar curvature in spectral sense.

    Notably, positive scalar curvature conditions do not control the number of ends of three-manifolds. The control on the number of ends in the theory of stable CMC hypersurfaces is usually derived with a non-trivial interaction between the Bochner formula for harmonic one-forms and the stability inequality, since the seminal work of Schoen and Yau \cite{schoen-yau-harmonicas}. We stress that this line of study has a different flavour in the minimal case as opposed to the CMC case, due to linear algebra inequalities that are harder to handle in the non-minimal case. As an example, we mention that it is known that a complete, connected, two-sided, stable minimal hypersurface in the Euclidean space must have one end, in every dimension $n\geq 3$, but the analogous result was proved to hold true for noncompact, weakly stable CMC hypersurfaces only up to dimension 7 (\cite{Cao-Shen-Zhu},\cite{Detang-XuCheng-Cheung},\cite{Fu-Li},\cite{miranda2026}). This linear algebra obstruction explains the threshold for rigidity observed by X. Cheng \cite{XuCheng} and Q. Deng \cite{Deng2011} in their study of stable CMC hypersurfaces in hyperbolic space (see also the appendix in \cite{miranda2026}). Theorem \ref{thm:main} shows that indeed there is a geometric obstruction for the control of the number of ends in lower regimes of mean curvature.

The number $H=n-1$ is a critical mean curvature value for the hyperbolic space $\mathbb{H}^n$ also in a different aspect. There exist compact CMC hypersurfaces with non-negative mean curvature $H$ embedded in $\mathbb{H}^n$ if and only if $H>n-1$. The existence of compact CMCs with mean curvature $H$ imposes restrictions on the geometry of non-compact CMCs with mean curvature $H$ due to the maximum principle (see, for instance, \cite{Manzano-Perez-Rodriguez} and \cite{SurveyNelli}). This was once regarded as evidence in favor of positive answers to both Hong’s and Chodosh’s questions.

In light of Theorem \ref{thm:main}, it is natural to ask: what is the least positive number $H_*^n$ such that there exists no complete strongly stable CMC hypersurface immersed in $\mathbb{H}^n$ with mean curvature $H>H_*^n$? Is the optimal value attained by some CMC hypersurface with distinguished properties? Note that $H_*^3 = 2$, and equality is attained uniquely by horospheres in this case \cite{daSilveira1987}.

\begin{aidisc}
    The idea of constructing stable CMC hypersurfaces with mean curvature greater than $3$ in $\HH^4$ by combining results of Mazzeo-Pacard \cite{MazzeoPacard2011}, Nayatani \cite{Nayatani1997} and Brooks \cite{Brooks1981} was suggested to the authors by ChatGPT 5.6 Sol. The authors subsequently developed and verified the argument, using the model for assistance with some intermediate computations, and the identification of relevant references. All mathematical statements, proofs, and citations were independently checked by the authors, who take full responsibility for the contents of the article. The final manuscript was written and revised by the authors.
\end{aidisc}

\begin{ack}
    We are grateful to Otis Chodosh for putting us in contact, for the discussions on the subject of this article and for his interest in the result. 
    We thank Franco Vargas Pallete for pointing us to \cite{analytic}. I. M. is also thankful to Lucas Ambrozio for his interest, Misha Belolipetsky for pointing references and Luis Zegarra for helpful conversations.
    F. G. was supported by the Swiss National Science Foundation (SNSF)
through the Postdoc.Mobility grant P500PT\_230344. I. M. was financed in part by the Coordenação de Aperfeiçoamento de Pessoal de Nível Superior - Brasil (CAPES) – Finance Code 001.
\end{ack}

\section{Technical setup}
A {\it conformally compact} space is a complete Riemannian manifold $(X,g)$ where $X$ is the interior of a smooth,
compact manifold with boundary $\overline{X}$ and the metric can be written as $g = \rho^{-2}\overline{g}$,
where $\overline{g}$ is a (nondegenerate) Riemannian metric on $\overline{X}$ and $\rho$ is a defining function for $\partial X$, i.e.,
$\rho$ is a smooth nonnegative function which vanishes only on $\partial X$, with $d\rho \neq 0$ there. Since 
$g = (a \rho)^{-2} (a^2 \overline g)$ for any smooth positive function $a$, we see that only the conformal class of the restriction 
of $\overline{g}$ to $T \partial X$ is well-defined. This is called the conformal infinity $\mathfrak c(g)$, and is a conformal structure
on the ideal boundary $\partial_\infty X$. 

Hyperbolic space $\mathbb H^n$ in its Poincar\'e ball model $(B^n, g)$ is the first example of a conformally compact space.
Here $\overline{g}$ is the Euclidean metric on the ball and $\rho(z) = (1-|z|^2)/2$.  There is a wider class of examples
of interest in this paper, namely the class of convex cocompact quotients of hyperbolic space.  
Let $\Gamma$ be a discrete, torsion-free subgroup of $\mathrm{Isom} \, \mathbb H^n$ which acts propertly discontinuously on 
$\mathbb H^n$.  Thus 
\[
X_\Gamma:=\mathbb H^n/\Gamma
\]
with the induced hyperbolic metric is smooth and complete, and we suppose that it is noncompact.  The limit set $\Lambda_\Gamma$
is, by definition, the set of accumulation points on $\partial_\infty \mathbb H^n$ of any orbit $\Gamma \cdot o$, where $o$ is any
basepoint in $\mathbb H^n$.  Its convex hull $\mathrm{conv}(\Lambda_\Gamma)$ in $\mathbb H^n$ is invariant with respect
to the action of $\Gamma$, and the quotient $\mathrm{conv}(\Lambda_\Gamma)/\Gamma$ is called the convex core of
$X_\Gamma$.  We then say that $\Gamma$ is convex cocompact if this convex core is compact.   

It turns out that any convex cocompact quotient is geometrically finite and conformally compact.  Indeed, define the 
domain of discontinuity $\Omega_\Gamma$ as the complement $S^{n-1}\setminus \Lambda_\Gamma$. The action of $\Gamma$
extends to a conformal action on $S^{n-1}$ and restricts to a properly discontinuous action on $\Omega_\Gamma$. We 
then realize the geodesic compactification of $X_\Gamma$ as the quotient
\[
\overline{X}_\Gamma = (\mathbb H^n \sqcup \Omega_\Gamma)/\Gamma, \qquad  \partial_\infty X_\Gamma :=  \Omega_\Gamma / \Gamma.
\]
The conformal infinity of the hyperbolic metric on $X_\Gamma$ is naturally induced from the standard conformal structure on the sphere.

If $(X,g)$ is any conformally compact manifold, and $h_0$ is a metric on $\partial_\infty X$ representing the conformal infinity of $g$, then
a well-known result by Graham and Lee \cite{GrahamLee1991} asserts that there exists a unique defining function $x$ for $\partial_\infty X$ in 
$\overline{X}$ defined in a collar neighborhood $\mathcal U$ of the boundary, such that in terms of the corresponding product 
decomposition $\mathcal U = [0, \epsilon_0) \times \partial X$, the metric $g$ takes the form
\[
g = \frac{dx^2 + h(x)}{x^2}.
\]
Here $h(x)$ is a smooth family of metrics on $\partial X$ with $h(0) = h_0$ the given representative.  Note that
$|dx/x|^2_g \equiv 1$, or equivalently, with $\overline{g} = x^2 g$, $|dx|^2_{\overline{g}} \equiv 1$. 

We now record a few geometric facts about this setup; these hold for $X_\Gamma$ with its hyperbolic metric, or more generally
for conformally compact spaces for which $g$ is Einstein (but we do not discuss that generalization further here). 
\begin{lemma}\label{lem: expansion-h}
If $n\geq 4$, then
\begin{align}\label{eq:lem-exp-Px}
h(x) =h_0-x^2 P_{h_0}+\frac{x^4}{4}P_{h_0}^2,
\end{align}
where 
\[
P_{h_0}
= \frac{1}{n-3} 
\left(
\operatorname{Ric}_{h_0}
-\frac{R_{h_0}}{2(n-2)}h_0
\right).
\]
denotes the Schouten tensor of $h_0$.
\end{lemma}
\begin{proof}
This is Theorem 7.4 in \cite{FeffermanGraham2012}.
\end{proof}
\begin{lemma}\label{lem: sff-levelset}
For $\varepsilon$ sufficiently small, the shape operators of the hypersurface $\{x=\varepsilon\}$ with respect to the metrics 
$g_\HH$ and $\overline g_\HH:=x^2g_\HH$, relative to the outward unit normals, are
\[
A_{\{x=\varepsilon\}}^{g_\HH}=I-\frac{\varepsilon}{2}h_{\varepsilon}^{-1}\partial_xh_x\vert_{x=\varepsilon},\qquad 
A_{\{x=\varepsilon\}}^{\overline {g_\HH}}=-\frac{1}{2}h_{\varepsilon}^{-1}\partial_xh_x\vert_{x=\varepsilon}
\]
Thus the corresponding mean curvatures are
\begin{align}\label{eq: MC-levelset}
H_{\{x=\varepsilon\}}^{g_\HH}=n-1+\varepsilon^2\operatorname{tr}(B_\varepsilon^{-1} h_0^{-1}P_{h_0}),\qquad 
H_{\{x=\varepsilon\}}^{\overline {g_\HH}}=\varepsilon\operatorname{tr}(B_\varepsilon^{-1} h_0^{-1}P_{h_0}),
\end{align}
where $B_x=I-\frac{x^2}{2}h_0^{-1}P_{h_0}$.
\end{lemma}
\begin{proof}
The outward unit normals with respect to $g_\HH$ and $\overline{g}_\HH$ are $\nu=-x\partial_x$ and $\overline{\nu}=-\partial_x$,
respectively. Let $X, Y$ be vector fields on $\{x=\varepsilon\}$ and extend them so that $[X,\partial_x]=[Y,\partial_x]=0$. Then by Koszul's formula
    \begin{align*}
2\overline g_{\HH}
\bigl(\nabla_X^{\overline g_{\HH}}\partial_x,Y\bigr)
={}&
X\bigl(\overline g_{\HH}(\partial_x,Y)\bigr)
+\partial_x\bigl(\overline g_{\HH}(Y,X)\bigr)
-Y\bigl(\overline g_{\HH}(X,\partial_x)\bigr)\\
&-\overline g_{\HH}\bigl(X,[\partial_x,Y]\bigr)
+\overline g_{\HH}\bigl(\partial_x,[Y,X]\bigr)
+\overline g_{\HH}\bigl(Y,[X,\partial_x]\bigr).
\end{align*}
Note that since $\partial_x$ is orthogonal to the level set, and by the choice of the extensions of $X,Y$, all but the second term vanish. Therefore
\[
A_{\{x=\varepsilon\}}^{\overline g_{\HH}}=-\frac 12h_\varepsilon^{-1}\partial_x h_x\vert_{x=\varepsilon}.
\]
    Now $g_\HH=e^{2f}\overline g_{\HH}$ with $f=-\log x$. Since $e^{-f}=x$ and $\overline{\nu}(f)=\frac{1}{x}$, by Lemma \ref{lem-sff-conformal}
    \[
    A^{g_\HH}_{\{x=\varepsilon\}}=e^{-f}(\overline{\nu}(f)I+A^{\overline g_{\HH}}_{\{x=\varepsilon\}})=I-\frac{\varepsilon}{2}h_x^{-1}\partial_x h_x\vert_{x=\varepsilon}.
    \]

    To compute the mean curvatures, note that \eqref{eq:lem-exp-Px} can be reformulated by saying that for any $X,Y$ tangent to $\{x=\varepsilon\}$
    \[
    h_\varepsilon(X,Y)=h_0(B_\varepsilon X, B_\varepsilon Y).
    \]
    Thus in a local frame $\det h_\varepsilon=\det^2(B_\varepsilon)\det h_0$. By Jacobi's formula
    \begin{align*}
\operatorname{tr}\left(
h_\varepsilon^{-1}
\left.\partial_xh_x\right|_{x=\varepsilon}
\right)
=
\left.\partial_x\log\det h_x\right|_{x=\varepsilon}=
2\operatorname{tr}\left(
B_\varepsilon^{-1}
\left.\partial_xB_x\right|_{x=\varepsilon}
\right)=
-2\varepsilon\operatorname{tr}\left(
B_\varepsilon^{-1}h_0^{-1}P_{h_0}
\right).
\end{align*}
    
\end{proof}

\subsection{A smooth family of Schottky groups}\label{ssec: S-groups}
We now specialize further from the general class of convex cocompact hyperbolic manifolds to the more special class
of Schottky manifolds.  These are defined by subgroups $\Gamma$ which are particularly simple to describe. 

Suppose that we are given $N$ pairs of points $(p_1,q_1), \ldots, (p_N, q_N)$ where each $p_i$ and $q_j$ lie
on $S^{n-1}$. Let $\gamma_j$ denote the hyperbolic geodesic which connects $p_j$ to $q_j$.  For each $j$,
choose totally geodesic subspaces $H_j$ and $K_j$ which intersect $\gamma_j$ orthogonally at two points
$p_j'$ and $q_j'$, respectively, so that $p_j, p_j', q_j', q_j$ appear in that order on $\gamma_j$.  The subspace
$H_j$ encloses a ball $B'_j \subset S^{n-1}$ centered at $p_j$, and similarly $K_j$ encloses a ball $B_j''$ centered
at $q_j$. We assume that these balls $B_j', B_j''$, $j = 1, \ldots, N$ are mutually disjoint. Let $T_j$ denote the geodesic 
distance along $\gamma_j$ between $p_j'$ and $q_j'$, and finally, choose a hyperbolic (or loxodromic) element $g_j$
which preserves $\gamma_j$ and maps $H_j$ to $K_j$.   In terms of all this data, we then define
\[
\Gamma = \langle  g_1, \ldots, g_N \rangle.
\]
This is called a Schottky group, and using a purely geometric argument, it is straightforward to see that it acts 
discretely and propertly discontinuously on $\mathbb H^n$, and that $\Gamma$ is a free group on $N$ generators. 
We may take as a fundamental domain for this action the intersection over $j$ of the regions between the 
hyperplanes $H_j$ and $K_j$. For simplicity, below we assume that all the $T_j$ are equal, and write this group as $\Gamma_{N,T}$. 

Given any discrete group of isometries on $\mathbb H^n$, we define its \textit{critical exponent} $\delta(\Gamma)$
\[
\delta(\Gamma)=\inf\left\{s>0:\sum_{\gamma\in \Gamma} e^{-sd_\HH(o,\gamma o)}<\infty\right\}.
\]
It is a beautiful and important fact in this theory, see by \cite[Theorem 1]{Sullivan1984}, that $\delta(\Gamma)$ equals the Hausdorff 
dimension of $\Lambda_\Gamma$. 

\begin{theorem}
\label{thm:schottky-groups}
If $n\geq 4$, there exists a family of such Schottky groups $\Gamma_{N,T}$, with $T$ lying in an open interval $\mathcal I$,
such that the associated map of critical exponents, $F: T \mapsto \delta(\Gamma_{N,T})$ is analytic. Furthermore, there exist 
some $T_0\in \mathcal I$ such that $F(T_0)=\frac{n-3}{2}$ and $F(T)-\frac{n-3}{2}$ changes sign at $T_0$.
\end{theorem}
\begin{proof}
To simplify matters even further, we assume that every one of the geodesics $\gamma_j$ passes through a fixed point $o \in
\mathbb H^n$, and that $p_j', q_j'$ are at equal distances $T/2$ from $o$. (The argument in the general case is very similar, but
requires one extra estimation.)  Also, write $\Gamma_{N,T}$ simply as $\Gamma$.

Let $g$ be any element in $\Gamma$. Then $g$ can be uniquely expressed as a reduced word in the generators $g_j$ and
their inverses.  There are precisely $2N(2N-1)^{k-1}$ such words of length $k$, and for each $g$ corresponding to
such a word, by the triangle inequality, the distance $d(o, g\cdot o)$ is no greater than $k T$.  This implies that
\[
\sum_{\gamma\in \Gamma} e^{-sd(o,\gamma o)}\geq \sum_{k=1}^\infty 2N(2N-1)^{k-1}e^{-sk{T}}.
\]
This sum diverges if $(2N-1)e^{-s T} \geq 1$, hence
\[
\delta(\Gamma) \geq \frac{ \log (2N-1)}{T}.
\]

We now show that it is possible to choose $N$ and $T$ so that $\log (2N-1)/T > (n-3)/2$, while on the other hand, if $N$ is fixed,
we show that $\delta(\Gamma) \to 0$ as $T \to \infty$.  As we explain later, $\delta$ is continuous, and in fact analytic, in $T$,
so this will imply the result. 

For the first of these inequalities, write the spherical radius of each of the balls $B_j', B_j''$ as $\theta$; this depends on $T$.
For that given $T$, choose $N$ so that the union of these balls occupies at least some definite fraction of the volume of the entire 
sphere, say $2N (\sin \theta)^{n-1} \geq c_n > 0$, which in turn gives that
\[
\log (2N) + (n-1)\log \theta \geq c_n'.
\]
Next,
a formula in spherical trigonometry and some basic hyperbolic geometry yields that
\[
T=2\log\cot(\theta/2)+1 =2\log(\theta^{-1})+\mathcal O(1).
\]
Putting these two estimates together gives
\[
\frac{\log (2N-1)}{T} \geq \frac{ (n-1) \log \theta^{-1} + \mathcal O(1)}{ 2 \log \theta^{-1} + \mathcal O(1)}.
\]
Choosing $T$ sufficiently large (and thus $\theta$ sufficiently small), we conclude that $\delta(\Gamma) > (n-3)/2$. 
Note that here we have first chosen $T$ and only after that, $N$.

The second assertion states that now if $N$ is fixed and we increase $T$ even further, then $\delta(\Gamma) \to 0$.  
For this we use the fact that $\delta(\Gamma)$ also equals the Hausdorff dimension of the limit set.   For Schottky groups,
the limit set is a Cantor set, and can be written as the countable intersection of the finite union of balls in $S^{n-1}$
centered at translates of the points $p_j$ and $q_j$. Indeed, write any group element $g$ in reduced form as 
$g_{i_1}...g_{i_k}$, where $g_{i_j}\in \{g_1, g_1^{-1},..., g_k,g_k^{-1}\}$. Now consider the ball $B$ among $B_{i_k}'$ and $B_{i_k}''$ centered at the sink of the transformation $g_{i_k}$, and apply successively $g_{i_{k-1}}$,..., $g_{i_1}$ to $B$.  Denote the resulting set by $K_g$.\\
We claim that there exists a constant $C > 0$
such that the diameter of each one of these balls is bounded by $(C e^{-T})^{k-1}$.  This is most easily seen as follows. 
We transform the ball to the upper half-space model of $\mathbb H^n$ and for a given $j$, position $p_j$ at $\infty$ and $q_j$
at $0$.  Then $g_j$ corresponds to the dilation $z \mapsto e^{-T} z$. All of the other balls $B_{j'}$ lie within a fixed large
ball $B_R \subset \mathbb R^{n-1}$, and so their volumes are contracted by a factor proportional to $e^{-T}$, where the
constant of proportionality is determined by $R$.  Since $R$ can be chosen independently of $j = 1, \ldots, N$, and
since the Jacobian factors for the transformation between the ball and the upper half-space are uniform within this ball
of radius $R$, we see that there is a uniform rate of volume contractivity (if $T$ is chosen to be large enough).  

Returning to the ball model, observe that given any $\epsilon > 0$ there exists a $k_0$ such that the diameter
of each component of $K_g$ is less than $\epsilon$ when $w(g) \geq k_0$, where $w(g)$ denotes the reduced
wordlength of $g$. 

Now, one can verify that
\[
\Lambda_\Gamma = \bigcap_{k \geq 1} \bigcup_{g:\, w(g) = k}  K_g.
\]
Using the calculations above, given any $d > 0$, we estimate the $d$-dimensional Hausdorff measure
\[
\mathcal{H}^d (\Lambda_{\Gamma})\leq \sum_{k \geq k_0} \sum_{g: w(g)=k} (\operatorname{diam}K_g)^d
\leq \sum_{k \geq k_0}  (2N (2N-1)^{k-1}) (C e^{-T})^{(k-1)d} = \sum_{k \geq k_0} C_1 C_2^{k-1} e^{ (-T + C_3)(k-1)d}.
\]
For any given $d > 0$ we can certainly choose $T > d^{-1}\log C_2  + C_3$, hence this sum is as small
as desired provided we choose $k_0$ sufficiently large.   In other words, the $d$-dimensional Hausdorff
measure of $\Lambda_\Gamma$ equals $0$, for $T$ large enough.

We now point out that in an appropriate sense, $\Gamma_{N,T}$ depends smoothly, and even analytically, on $T$.
Indeed, denoting by $e_1, \ldots e_N$ the generators of the free group $F_N$, then $\Gamma_{N,T}$ is uniquely
determined by the assignments of elements $g_j \in \mathrm{Isom}\, \mathbb H^n$ to each $e_j$. The 
space of Schottky representations is thus an open subset of $(\mathrm{Isom}\, \mathbb H^n)^N$ and varying
$T$ certainly produces an analytic variation of this element. 

Our final claim is that $\delta(\Gamma_{N,T})$ is analytic in $T$.  For this we refer to Corollary 1.8 in \cite{analytic}.

Although we do not claim that $\delta(\Gamma)$ is monotone in $T$, the analyticity guarantees an isolated
crossing of the value $(n-3)/2$ at some $T_0$. 
\end{proof}

We conclude this section with the following
\begin{lemma}\label{lem:iso-p1}
For any Schottky group $\Gamma$ on $\mathbb H^n$, $n \geq 4$, the inclusion 
\[
Y_\Gamma := \Omega/\Gamma \hookrightarrow \overline{X}_\Gamma
\]
induces an isomorphism
\[
\iota_*: \pi_1(Y)  \longrightarrow \pi_1(\overline{X}_\Gamma).
\] 
\end{lemma}
\begin{proof}
Since the fundamental domain of $\Gamma$ is equal to the ball $B^n$ minus a finite number of half-spaces, and
similarly, $Y_\Gamma$ is the complement of a finite number of disjoint balls in $S^{n-1}$, it follows that any closed loop
in $X_\Gamma$ may be homotoped to a closed loop in $Y_\Gamma$. 
\end{proof}

\section{Proof of the main Theorem}
\begin{theorem}\label{thm: cmc-in-quotient}
There exists an open interval $I$ containing $T_0$ and $\varepsilon_0>0$ such that for $\varepsilon\in (0,\varepsilon_0)$, $T\in I$ 
there exist Schottky groups $\Gamma_T$ and connected, two sided hypersurfaces $\Sigma_{\varepsilon, T}$ in $X_T=\HH^n/{\Gamma}_T$ 
such that
\begin{enumerate}
\item $\Sigma_{\varepsilon, T}$ has constant mean curvature $H_{\varepsilon, T}$, and $T\mapsto H_{\varepsilon, T}$ takes values 
in a fixed neighborhood of $n-1$.
\item $\Sigma_{\varepsilon, T}$ is graphical over $Y$, and the pullback metric $\hat g_{\varepsilon,T}$ on $Y$ satisfies
\[
\hat g_{\varepsilon,T}=\varepsilon^{-2}(h_T+O_{C^0}(\varepsilon^2))
\]
for a family of metrics $h_T$ depending smoothly on $T$.
\item \[
\lVert\lvert A_{\Sigma_{\varepsilon, T}}\rvert^2-(n-1)\rVert_{L^\infty(\Sigma_{\varepsilon, T})}\leq C\varepsilon^2\lvert T-T_0\rvert+o(\varepsilon^2)
\]
for some constant $C>0$ and error $o(\varepsilon^2)$ independent from $T$.
\end{enumerate}
\end{theorem}
\begin{proof}
Henceforth $N$ is fixed and will be dropped from the notation. The spaces $X_{\Gamma_{T}}$ are mutually diffeomorphic, so 
there exists a smooth family of diffeomorphisms $\Psi_T$, with $\Psi_{T_0}$ equal to the identity, such that $g_T = 
\Psi_T^* g_{\mathrm{hyp}}$ is a smooth family of conformally compact hyperbolic metrics on $X = X_{T_0}$. The corresponding
conformal infinities $\mathfrak c(g_T)$ are a smooth family of conformal classes on the fixed manifold $Y = \partial X$. 

We now quote Theorem 3.3 in \cite{Nayatani1997}, which states that the Yamabe constant of $(Y_T, \mathfrak c(g_{\mathrm{hyp}}))$ 
(and hence of $(Y, \mathfrak c(g_T))$) has the same sign as
\begin{align}\label{eq: sign-Nayatani}
\frac{n-3}{2}-\delta(\Gamma_T).
\end{align}
By Lemma \ref{lem: Yamabe-metrics} below, there exists a family of metrics $h_T\in \mathfrak c_T$ on $Y$, depending smoothly 
on $T$ for $T$ near $T_0$, such that $h_T$ has constant scalar curvature $R_T$, where $R_T$ has the same sign as \eqref{eq: sign-Nayatani}.
Denote by $x_T$ the family of associated geodesic defining functions on $\overline{X}$. These too depend smoothly on $T$, and
for $T$ sufficiently close to $T_0$, define a common collar of $Y$. Fixing $k\in \mathbb N$ and $\alpha\in (0,1)$, set
\[
C^{k,\alpha}_0(Y):=\left\{u\in C^{k,\alpha}(Y): \int_Yu\, d\mu_{h_{T_0}}=0\right\}.
\]
For $\varepsilon>0$ and $w\in C^{2,\alpha}_0(Y)$ sufficiently small, denote by $\Sigma_{T,\varepsilon,w}$ the hypersurface given 
in this collar neighborhood given as the graph of the function
\[
x_T(y)=\varepsilon e^{-w(y)}.
\]
The mean curvature of $\Sigma_{T,\varepsilon,w}$ is a function $H(T,\varepsilon,w)$, and we write
\[
\mathcal{N}(T,\varepsilon,w):=\frac{n-1-H(T,\varepsilon,w)}{\varepsilon^2}.
\]
By Lemma \ref{lem: N-operator}, $\mathcal{N}$ depends smoothly on its variables, and extends smoothly to $\varepsilon=0$.
    For any integrable function $f$ on $Y$, set
    \[
    \Pi f:=f-\fint_Y fd\mu_{h_{T_0}}
    \]
    and set
    \[
    \mathcal{F}(T, \varepsilon,w):=\Pi\mathcal N(T, \varepsilon,w).
    \]
    Then $\Sigma_{(T, \varepsilon,w)}$ has constant mean curvature if and only if $\mathcal{F}(T, \varepsilon,w)=0$.
    By Lemma \ref{lem: N-operator}, $\mathcal{F}$ also extends smoothly to $\varepsilon=0$, and
    \[
    \mathcal{F}(T_0,0,0)=0
    \]
    Moreover,
    \begin{align}
    \label{eq: expansion-F}
        D_w\mathcal{F}(T_0,0,0)[v]=\Delta_{h_{T_0}}v.
    \end{align}
    This is an isomorphism $C_0^{2,\alpha}(Y)\to C^{0,\alpha}_0(Y)$.
    By the implicit function theorem (applied to any smooth extension of $\mathcal F$ around $(T_0,0)$), there exists a neighborhood $U$ of $(T_0,0)$ in $\mathbb R\times\mathbb R_{\geq 0}$ and a family
    \begin{align}
        w_{T,\varepsilon}\in C_0^{2,\alpha},
    \end{align}
    depending smoothly on $(T,\varepsilon)\in U$, such that
    \[
    \mathcal{F}(T,\varepsilon,w_{T,\varepsilon})=0.
    \]
    Note that by Lemma \ref{lem: N-operator}, $\mathcal F(T,0,0)=0$, therefore $w_{T,0}=0$.
Now set
\[
c(T,\varepsilon):=\mathcal{N}(T,\varepsilon,w_{T,\varepsilon}).
\]
As the expression on the right is constant, we can regard $c$ as a continuous real valued function on $U$.
Comparing the definition of $\mathcal{N}$ with \eqref{eq: MC-levelset} we see that
\begin{align}
    \label{eq: expansion-N-e}
    \mathcal{N}(T,\varepsilon,0)=&-\operatorname{tr}\left(\left(I-\frac{\varepsilon^2}{2}h_T^{-1}{P_T}\right)^{-1}h_T^{-1}{P_T}\right).
\end{align}
Then
\[
c(T,0)=-\operatorname{tr}(h_T^{-1}{P_T})=-\frac{R_T}{2(n-2)}.
\]
Now choose numbers $T_-$ and $T_+$ with $T_-<T_0<T_+$, sufficiently close to $T_0$ and choose $\varepsilon_0>0$ such that
\[
[T_-, T_+]\times [0,\varepsilon_0]\subset U
\]
and for $\varepsilon\in (0,\varepsilon_0)$
\[
c(T_-, \varepsilon)c(T_+,\varepsilon)<0.
\]
For this, we use the fact that $\delta(\Gamma_T)$ is a non-constant analytic function of $T$, and $\delta(\Gamma_T)-c_n$ changes sign at $T=T_0$. Since $R_T$ has the same sign as \eqref{eq: sign-Nayatani}, one can find such $T_-$ and $T_+$.\\
Therefore, the hypersurfaces $\Sigma_{T,\varepsilon,w_{T,\varepsilon}}$ for $T\in (T_-,T_+)$ have constant mean curvature $H_{T,\varepsilon}$ varying continuously in a fixed interval containing $n-1$. This proves the first assertion.\\
Next, note that expanding \eqref{eq: expansion-N-e} in $\varepsilon$ and applying $\Pi$ we obtain
    \[
    \mathcal{F}(T,\varepsilon,0)=-\frac{\varepsilon^2}{2}\Pi\lvert P_{h_T}\rvert^2+O_{C^{0,\alpha}}(\varepsilon^4).
    \]
    Then for any $T\in I$, $\partial_\varepsilon\mathcal{F}(T,0,0)=0$.
    Therefore by \eqref{eq: expansion-F}
    \[
    0=\partial_\varepsilon\mathcal{F}(T,\varepsilon,w_{T,\varepsilon})\vert_{\varepsilon=0}=D_w\mathcal{F}(T,0,0)[\partial _\varepsilon w_{T,\varepsilon}]=\Pi\left(\Delta_{h_{T}}\partial_{\varepsilon}w_{T,\varepsilon}+\frac{R_T}{n-2}\partial_{\varepsilon}w_{T,\varepsilon}\right).
    \]
    Since $w_{T,\varepsilon}$ has zero average, so does
$\left.\partial_\varepsilon w_{T,\varepsilon}\right|_{\varepsilon=0}$.
As the operator
$L_T:C^{2,\alpha}_0(Y)\to C^{0,\alpha}_0(Y)$
is invertible for $T$ sufficiently close to $T_0$,
we deduce that $\partial_\varepsilon w_{T,\varepsilon}
|_{\varepsilon=0}=0.$
Therefore
    \begin{align}
        \label{eq: w-quadratic}
    \lVert w_{T,\varepsilon}\rVert_{C^{2,\alpha}(Y)}=O(\varepsilon^2).
    \end{align}
Now write the hyperbolic metric locally as
\begin{align}\label{eq: local-hyp-metric-S}
g_T=x_T^{-2}(dx_T^2+h_{T,x}),\quad h_{T,x}=h_T-x_T^2P_{h_T}+\frac{x_T^4}{4}P_{h_T}^2
\end{align}
(the second identity was obtained in Lemma \ref{lem: expansion-h}). 
Writing the graph map as $F_{T,\varepsilon}(y)=(r_{T,\varepsilon}(y), y)$, we have
\[
F_{T,\varepsilon}^\ast g_T=r_{T,\varepsilon}^{-2}(dr_{T,\varepsilon}^2+h_{T,r_{T,\varepsilon}})=\varepsilon^{-2}e^{2w_{T,\varepsilon}}h_{T,r_{T,\varepsilon}}+dw_{T,\varepsilon}^2.
\]
Note that
\[
e^{2w_{T,\varepsilon}}h_{T,r_{T,\varepsilon}}- h_T=O(\varepsilon^2)
\]
for $\varepsilon\to 0$ by \eqref{eq: w-quadratic} and the second expression in \eqref{eq: local-hyp-metric-S}, therefore the second assertion in the Theorem holds.\\
To prove the third assertion, write the metric of the ideal compactification of $(X, g_T)$ as
\[
\overline{g_T}=x_T^2 g_T=dx_T^2+h_{T,x}
\]
By Lemma \ref{lem: sff-levelset}, the shape operator of the level set $\{x_T=\varepsilon\}$ with respect to $\overline{ g_T}$ is given by
\[
\overline{A}_{\{x_T=\varepsilon\}}=-\frac{1}{2}h_{T,\varepsilon}^{-1}\partial_x h_{T,x}\vert_{x=\varepsilon}=\varepsilon h_T^{-1}P_{h_T}+O(\varepsilon^3).
\]
Since $\{x_T=\varepsilon\}$ and $\Sigma_{T,\varepsilon}$ are graph over $Y$ for functions whose difference is $O_{C^2}(\varepsilon^3)$, we have that the shape operator of $\Sigma_{T,\varepsilon}$ with respect to $\overline{g_T}$ is
\[
\overline A_{\Sigma_{T,\varepsilon}}=\varepsilon h_T^{-1}P_{h_T}+O(\varepsilon^3).
\]
Now note that $\overline{\nu}$, the outward pointing unit normal to $\Sigma_{T,\varepsilon}$ with respect to $\overline g_T$, satisfies
\[
    \overline{\nu}(x_T)=-(1+\lvert dr_{T,\varepsilon}\rvert^2_{h_{T, r_{T,\varepsilon}}})^{-\frac{1}{2}}.
\]
By \eqref{eq: w-quadratic}, there holds $\lvert dr_{T,\varepsilon}\rvert_{h_{T, r_{T,\varepsilon}}}=O(\varepsilon^3)$, and thus $\overline \nu(x_T)=-1+O(\varepsilon^6)$.
Therefore by Lemma \ref{lem-sff-conformal}
\[
A_{\Sigma_{T,\varepsilon}}=-\overline \nu (x_T)I+x_T\overline A_{T,\varepsilon}=I+\varepsilon^2h_T^{-1}P_{h_T}+O(\varepsilon^4).
\]
Hence
\[
\lvert A_{\Sigma_{T,\varepsilon}}\rvert^2=n-1+\frac{R_T}{n-2}\varepsilon^2+O(\varepsilon^4).
\]
Since $R_T=O(\lvert T-T_0\rvert)$, the statement follows.
\end{proof}
For $T\in I$ let
\[
\pi_T:\HH^n\to X_{T}
\]
denote the natural quotient map, and for $\varepsilon\in (0,\varepsilon_0)$ set
\[
\tilde \Sigma_{\varepsilon,T}:=\pi_T^{-1}(\Sigma_{\varepsilon,T}).
\]
Note that since $\pi_T$ is a local isometry, $\tilde \Sigma_{\varepsilon,T}$ has constant mean curvature $H_{\varepsilon, T}$ in $\HH^n$, with $H_{\varepsilon, T}$ equal to the mean curvature of $\Sigma_{\varepsilon,T}$ in $X_T$.
\begin{theorem}
\label{thm:lift}
    If $I$ and $\varepsilon_0$ are chosen to be sufficiently small, the hypersurface $\tilde \Sigma_{\varepsilon,T}$ is simply connected, properly embedded, two-sided and complete, and it is strictly strongly stable.
    Moreover, it is not totally umbilical and has infinitely many ends.
\end{theorem}
\begin{proof}
For fixed $\varepsilon$ and $T$ set $\Sigma:=\Sigma_{\varepsilon,T}$ and $\tilde\Sigma:=\tilde\Sigma_{\varepsilon,T}$.
    As $\Sigma$ is a graph over $Y$, the inclusion of $\Sigma$ in $\overline{X}_T$ is homotopic to the inclusion of $Y$ in $\overline{X}_T$, therefore by Lemma \ref{lem:iso-p1}, the induced map
    \[
    (\iota_{\Sigma})_\ast:\pi_1( \Sigma)\to \pi_1(\overline X_T)\simeq\pi_1(X_T)
    \]
    is an isomorphism. The surjectivity of this map implies that $\tilde \Sigma$ is connected, while the injectivity implies that $\tilde \Sigma$ is simply connected. Thus $\tilde\Sigma$ is a universal cover of $\Sigma$.
    Since $\pi_T|_{\widetilde\Sigma}:\widetilde\Sigma\to\Sigma$ is a
Riemannian covering and $\Sigma$ is complete, the induced metric on
$\widetilde\Sigma$ is complete. Moreover, a global unit normal to
$\Sigma$ lifts to a global unit normal to $\widetilde\Sigma$, so
$\widetilde\Sigma$ is two-sided. As $\Sigma$ is embedded in $X$ as a graph over $Y$, and $\pi_T$ is a covering map, $\tilde \Sigma$ is embedded in $\HH^n$. Moreover, since \(\Sigma\) is closed in \(X_T\),
\(\widetilde\Sigma\) is closed in \(\mathbb H^n\). Hence
\(\widetilde\Sigma\) is properly embedded.\\
Next we show that $\tilde \Sigma$ is strictly strongly stable. Let $q:\tilde Y\to Y$ be the universal covering of $Y$. Then the graph map $F_{\varepsilon, T}: Y\to \Sigma$ induces a diffeomorphism
\[
\tilde F_{\varepsilon, T}: \tilde Y\to \tilde \Sigma.
\]
Denote by $g_{\tilde \Sigma}$ the metric induced on $\tilde \Sigma$ by $g_{\HH}$ (for this metric, $\pi_T \vert_{\tilde \Sigma}$ is a local isometry).
By property $(2)$ in Theorem \ref{thm: cmc-in-quotient}, the pull-back metric $g_{\varepsilon, T}:= (\tilde F_{\varepsilon, T})^\ast g_{\tilde \Sigma}$ is of the form
\[
g_{\varepsilon, T}=\varepsilon^{-2}q^\ast h_T+O(1).
\] 
We set
\begin{align}\label{eq: cpct-metric-conv}
    {\gamma}_{\varepsilon, T}:=\varepsilon^2g_{\varepsilon, T}=q^\ast h_T+O(\varepsilon^2).
\end{align}
For a complete Riemannian manifold \((M,g)\), let
\[
\lambda_0(M,g)
:=
\inf_{\substack{f\in C_c^\infty(M)\\f\neq0}}
\frac{
\displaystyle\int_M|\nabla^g f|_g^2\,d\mu_g
}{
\displaystyle\int_M f^2\,d\mu_g
}
\]
denote the bottom of the spectrum of its Laplacian. Then
\begin{align}\label{eq: equivalence-lambda0}
\lambda_0(\tilde \Sigma, g_{\tilde\Sigma})=\lambda_0(\tilde Y,  g_{\varepsilon, T})=\varepsilon^2\lambda_0(\tilde Y, {\gamma}_{\varepsilon, T}).
\end{align}
We will now show that
\begin{align}
    \label{eq: lower-bound-spectrum}
    \lambda_0(\tilde Y,q^\ast h_{T_0})>c_0>0
\end{align}
for some constant $c_0$. By \eqref{eq: cpct-metric-conv} and \eqref{eq: equivalence-lambda0}, this will imply that, for $I$ and $\varepsilon_0$ sufficiently small,
\begin{align}\label{eq: lower-bound-lambda-Sigma}
    \lambda_0(\tilde \Sigma, g_{\tilde \Sigma})>\frac{c_0}{2}\varepsilon
    ^2.
\end{align}
To see that \eqref{eq: lower-bound-spectrum} holds, recall that $(\tilde Y, q^\ast h_{T_0})$ is a Riemannian universal cover of $(Y, h_{T_0})$, and
\[
\pi_1(Y)\simeq \pi_1(X)\simeq F_N,
\]
for some $N\geq 2$.
Therefore $\pi_1(Y)$ is non-amenable. By \cite[Theorem 1]{Brooks1981}, this implies that $\lambda_0(\tilde Y,q^\ast h_{T_0})>0$.\\
Now for any $f\in C^\infty_c(\tilde \Sigma)$, by \eqref{eq: lower-bound-lambda-Sigma} and $(3)$ in Theorem \ref{thm: cmc-in-quotient},
\begin{align*}
    Q(f)=&\int_{\tilde \Sigma}(\lvert\nabla f\rvert^2-(\lvert A_{\tilde \Sigma}\rvert^2-(n-1))f^2)d\mu_{g_{\tilde \Sigma}}\\
    \geq &\int_{\tilde \Sigma}(\lambda_0(\tilde \Sigma, g_{\tilde \Sigma})-\varepsilon^2C\lvert T-T_0\rvert+o(\varepsilon^2))f^2 d\mu_{g_{\tilde \Sigma}}\\
    \geq &\varepsilon^2\left(\frac{c_0}{2}-C\lvert T-T_0\rvert+o_{\varepsilon}(1)\right)\int_{\tilde \Sigma}f^2 d\mu_{g_{\tilde \Sigma}}.
\end{align*}
If $I$ and $\varepsilon_0$ are chosen sufficiently small, the expression in brackets is larger than $\frac{c_0}{4}$. This completes the proof of strict strong stability.\\
The group $\Gamma_T$ acts properly, cocompactly, and isometrically on $\widetilde\Sigma_{\varepsilon,T}$. Therefore, by the Švarc-Milnor lemma \cite[Proposition I.8.19]{BridsonHaefliger1999}, $\widetilde\Sigma_{\varepsilon,T}$ is quasi-isometric to a Cayley graph of \(\Gamma_T\simeq F_N\) with $N\geq 2$.
Since $F_N$ has infinitely many ends and the number of ends is invariant under quasi-isometries of proper geodesic spaces \cite[Proposition I.8.29]{BridsonHaefliger1999}, $\widetilde\Sigma_{\varepsilon,T}$ has infinitely many ends.
Therefore, since simply connected, complete, totally umbilical hypersurfaces in $\mathbb{H}^n$ are classified and have at most one end, we conclude that $\tilde \Sigma$ is not umbilical.
\end{proof}

\begin{remark}
A slightly simpler argument can be given for the existence of CMC hypersurfaces in $\HH^n$ with mean curvature $H$ either strictly greater or strictly smaller than $n-1$, and with $|H-(n-1)|$ arbitrarily small. In fact, one can choose a group $\Gamma_T$ as in Theorem \ref{thm:schottky-groups} such that the corresponding constant scalar curvature $R_T$ of $Y$ is small and has the chosen sign (see the proof of Theorem \ref{thm: cmc-in-quotient}). For this fixed $T$, one can construct CMC hypersurfaces by deforming the level sets $\{x_T=\varepsilon\}$ for all sufficiently small $\varepsilon$, as above.
Alternatively, one can apply Theorem 5.1 in \cite{MazzeoPacard2011}.
The resulting mean curvature $H_\varepsilon$ is given by \eqref{eq: H(e,t)}, so $H_\varepsilon-(n-1)$ has the chosen sign for small $\varepsilon$. Since $H_\varepsilon$ depends continuously on $\varepsilon$ and tends to $n-1$ for $\varepsilon\to 0$, these hypersurfaces realize every
mean curvature in an  interval adjacent to $n-1$. Choosing the
opposite sign of $R_T$ covers an interval on the other side of $n-1$. Using Theorem \ref{thm:lift}, one can then show that the lift of the hypersurface satisfies the desired properties.
    This argument allows one to avoid the analysis of the dependence of the construction on $T$.
\end{remark}

\appendix

\section{Technical computations}
We collect here a few technical results. See the proof of Theorem \ref{thm: cmc-in-quotient} for notation.
\begin{lemma}\label{lem: Yamabe-metrics}
    There exists a family of metric $h_T\in \mathfrak{c}_T$ on $Y$, varying smoothly on $T$ for $T$ around $T_0$, such that $h_T$ has constant scalar curvature $R_T$, where $R_T$ has the same sign as \eqref{eq: sign-Nayatani}.
\end{lemma}
\begin{proof}
Let $h_0\in \mathfrak{c}_{T_0}$ be the scalar flat metric of volume one on $Y$
    For an interval $I$ around $T_0$, for $T\in I$ let $k_T$ be a representative of $\mathfrak c_T$, varying smoothly on $T$, and such that $k_{T_0}=h_0$.
    Any element $h$ of $\mathfrak c_T$ can be written as $h=u^{\frac{4}{n-3}} k_T$ for some positive function $u$. Finding a metric $h_T\in \mathfrak c_T$ with constant scalar curvature $R_T$ and volume $1$ is then equivalent to find a solution $u_T$ of
    \begin{align}\label{eq: Yamabe-equation}
    -\frac{4(n-2)}{n-3}\Delta_{k_T}u_T+\operatorname{Scal}_{k_T}u_T=R_T u_T^\frac{n+1}{n-3},\quad \int_Yu_T^\frac{2(n-1)}{n-3}d\mu_{k_T}=1.
    \end{align}
    To this end, for any $\alpha\in (0,1)$ we define the operator
    \begin{align*}
    \mathcal{Y}:&I\times C^{2,\alpha}_+(Y)\times\mathbb R\to C^{0,\alpha}(Y)\times \mathbb R,\\
    &(T,u,R)\mapsto\left(-\frac{4(n-2)}{n-3}\Delta_{k_T}u+\operatorname{Scal}_{k_T}u-R u^\frac{n+1}{n-3}, \int_Yu^\frac{2(n-1)}{n-3}d\mu_{k_T}-1\right),
    \end{align*}
    where $C^{2,\alpha}_+(Y)=\{u\in C^{2,\alpha}(Y): u>0\}$.
    Since $k_T$ depends smoothly on $T$, $\Delta_{k_T}$, $\operatorname{Scal}_{k_T}$ and $d\mu_{k_T}$ also depend smoothly on $T$, therefore $\mathcal{Y}$ is a smooth map.
    Note that
    \[
    \mathcal{Y}(T_0,1,0)=0
    \]
    and
    \[
    D_{(u,R)}\mathcal{Y}(T_0,1,0)[v,\rho]=\left(-\frac{4(n-2)}{n-3}\Delta_{h_0}v-\rho, \frac{2(n-1)}{n-3}\int_Y v\,d\mu_{ h_0}\right).
    \]
    Note that this operator is invertible as a map $C^{2,\alpha}(Y)\times\mathbb{R}\to C^{0,\alpha}(Y)\times\mathbb{R}$: integrating its first component determines $\rho$, $\Pi v$ is then determined by the invertibility of the Laplacian on mean-zero functions, and the second
component determines $v-\Pi v$.
    Therefore, by the smooth implicit function theorem there exists a smooth map
    \[
    I\to C^{2,\alpha}_+(Y)\times \mathbb R,\quad T\mapsto (u_T, R_T)
    \]
    (up to choosing $I$ smaller if necessary) such that for any $T\in I$ $\mathcal{Y}(T,u_T, R_T)=0$.
    Thus for any such $T$
    \[
    h_T:=u_T^{\frac{4}{n-3}}k_T
    \]
    is a metric on $Y$ with constant scalar curvature $R_T$.
    Since $u_T$ solves \eqref{eq: Yamabe-equation}, $u_T$ is actually smooth and depends smoothly on $T$, so the same holds for $h_T$.
    As $h_T\in \mathfrak{c}_T$, the sign of $R_T$ corresponds to the sign of the Yamabe constant of $(Y,\mathfrak c_T)$, which is given by \eqref{eq: sign-Nayatani}.    
\end{proof}
\begin{lemma}
    \label{lem: N-operator}
    The operator \(\mathcal N\) depends smoothly on its variables and
extends smoothly to \(\varepsilon=0\).\\
    Moreover
    \begin{equation}
\label{eq: explicit-N}
\mathcal N(T,0,w)
=
e^{-2w}
\left(
-\frac{R_T}{2(n-2)}
+\Delta_{h_T}w
+\frac{n-3}{2}|dw|_{h_T}^2
\right).
\end{equation}
\end{lemma}
\begin{proof}
Set
$r=\varepsilon e^{-w}$. In what follows, \(E\) denotes a tensor depending smoothly on
\(
    (T,r^2,y,w,Dw,D^2w)
\)
and which may change from line to line.
 The metric of the ideal compactification is given by $\overline g_T=dx_T^2+h_{T,x_T}$, where $h_{T,x_T}$ has the explicit form described in Lemma \ref{lem: expansion-h}.
The metric induced by \(\overline g_T\) on the graph \(x_T=r(y)\) is
\[
    \overline\gamma=h_{T,r}+dr\otimes dr.
\]
Since \(dr=-r\,dw\), it follows that
\[
    \overline\gamma
    =
    h_T+r^2\bigl(dw\otimes dw-P_{h_T}\bigr)+O(r^4),
\]
and hence
\[
    \overline\gamma^{-1}=h_T^{-1}+r^2E_1,
\]
where \(E_1\) depends smoothly on
\((T,\varepsilon,y,w,Dw)\).

The outward unit normal vector (with respect to \(\overline g_T\)) is
\[
    \overline\nu
    =
    \frac{-\partial_{x_T}+\nabla^{h_{T,r}}r}{W},
    \qquad
    W=\sqrt{1+|dr|_{h_{T,r}}^2}.
\]

A direct computation in the coordinates \((x_T,y)\) gives
\begin{align}
\overline{\mathbb{II}}_{ij}
=\frac1W\bigg(
&r_{ij}
-\Gamma^k_{ij}(h_{T,x_T})\big|_{x_T=r}r_k
-\frac12\partial_{x_T}(h_{T,x_T})_{ij}\big|_{x_T=r}
\notag\\
&-\frac12r_jr^\ell
  \partial_{x_T}(h_{T,x_T})_{i\ell}\big|_{x_T=r}
-\frac12r_ir^\ell
  \partial_{x_T}(h_{T,x_T})_{j\ell}\big|_{x_T=r}
\bigg),
\label{eq: graph-II}
\end{align}
where \(r^\ell=(h_{T,r})^{\ell k}r_k\).
Since $h_{T,x_T}$ is smooth in \(x_T^2\), we have
\[
    \Gamma^k_{ij}(h_{T,x_T})
    =
    \Gamma^k_{ij}(h_T)+x_T^2E^k_{ij},
    \qquad
    -\frac12\partial_{x_T}h_{T,x_T}
    =
    x_TP_{h_T}+x_T^3E.
\]
Moreover, since \(r_i=-rw_i\),
\begin{align*}
r_{ij}
-\Gamma^k_{ij}(h_{T,x_T})\big|_{x_T=r}r_k
=
(\nabla_{h_T}^2r)_{ij}+r^3E_{ij}=
r\left(
    w_iw_j-(\nabla_{h_T}^2w)_{ij}
\right)+r^3E_{ij}.
\end{align*}
Since $r_i$, $r^\ell$, and $\partial_{x_T}h_{T,t}\big|_{x_T=r}$ all are of order $O(r)$, the last two terms in \eqref{eq: graph-II} can be rewritten as $r^3E_{ij}$.
Finally,
\[
    W^{-1}
    =
    \left(1+r^2|dw|_{h_{T,r}}^2\right)^{-1/2}
    =
    1+r^2E.
\]
Substituting these identities into \eqref{eq: graph-II} gives
\[
    \overline{\mathbb{II}}
    =
    r\left(
        P_{h_T}-\nabla_{h_T}^2w+dw\otimes dw
    \right)
    +r^3E,
\]
so that the corresponding mean curvature is given by
\[
\overline{H}=r\left(\operatorname{tr}_{h_T}P_{h_T}-\Delta_{h_T}w+\lvert dw\rvert_{h_T}^2\right)+r^3 E=r\left(\frac{R_T}{2(n-2)}-\Delta_{h_T}w+\lvert dw\rvert_{h_T}^2\right)+r^3 E.
\]

Since $\overline \nu (x_T)=\frac{-1}{W}$, Lemma \ref{lem-sff-conformal} implies that
\[
    H(T,\varepsilon,w)
    =
    r\overline H+\frac{n-1}{W}.
\]
Moreover, $W^2=    1+r^2|dw|_{h_{T,r}}^2$, therefore
\[
    (n-1)\left(1-\frac1W\right)
    =
    \frac{n-1}{2}r^2|dw|_{h_T}^2+r^4E.
\]
Consequently,
\begin{align*}
    n-1-H(T,\varepsilon,w)
    =
    (n-1)\left(1-\frac1W\right)-r\overline H=
    r^2\left(
        -\frac{R_T}{2(n-2)}
        +\Delta_{h_T}w
        +\frac{n-3}{2}|dw|_{h_T}^2
        +r^2E
    \right).
\end{align*}
Thus
\[
\mathcal N(T,\varepsilon,w)
=
e^{-2w}\left(
    -\frac{R_T}{2(n-2)}
    +\Delta_{h_T}w
    +\frac{n-3}{2}|dw|_{h_T}^2
    +\varepsilon^2e^{-2w}E
\right).
\]
Since \(E\) is smooth in all its arguments, the right-hand side can be
written as
\[
    \Phi\bigl(
        T,\varepsilon^2,y,
        w(y),Dw(y),D^2w(y)
    \bigr)
\]
for a smooth function \(\Phi\). The associated Nemytskii operator is
smooth on Hölder spaces. Consequently, for \(\varepsilon_0>0\)
sufficiently small, and for a sufficiently small open neighborhood $\mathcal U$
of $0$ in $C^{2,\alpha}_0(Y)$,
\[
\mathcal N:
I\times(0,\varepsilon_0)\times\mathcal U
\longrightarrow C^{0,\alpha}(Y)
\]
admits a smooth extension to \(\varepsilon=0\).
Setting \(\varepsilon=0\) gives
\[
\mathcal N(T,0,w)
=
e^{-2w}
\left(
    -\frac{R_T}{2(n-2)}
    +\Delta_{h_T}w
    +\frac{n-3}{2}|dw|_{h_T}^2
\right).
\]
\end{proof}

\begin{lemma}
\label{lem-sff-conformal}
Let $\Sigma$ be a two-sided hypersurface in a Riemannian manifold $(M, g)$. Let $\widehat g=e^{2f}g$ for some function $f\in C^1(M)$. Let \(\nu\) and \(\widehat\nu\) be the
unit normals with respect to \(g\) and \(\widehat g\), respectively,
chosen with the same orientation. Then $\hat \nu=e^{-f}\nu$ and the respective shape operators are related as follows:
\[
A^{\widehat g}=e^{-f}(
\nu(f)I+A^g).
\]
\end{lemma}

\begin{proof}
The identity $\hat \nu=e^{-f}\nu$ is immediate.
By Koszul formula, the Levi-Civita connections of \(g\) and \(\widehat g\) satisfy, for any $C^1$ vector fields $X$, $Y$,
\[
    \nabla_X^{\widehat g}Y
    =
    \nabla_X^gY
    +X(f)Y+Y(f)X-g(X,Y)\nabla^gf.
\]
Thus, for \(X\in T\Sigma\),
\begin{align*}
\nabla_X^{\widehat g}\widehat\nu=
\nabla_X^{\widehat g}(e^{-f}\nu)=
e^{-f}\left(\nabla_X^g\nu+\nu(f)X\right),
\end{align*}
since \(g(X,\nu)=0\). As $A^g(X)=\nabla^g_X\nu$ and $A^{\widehat g}(X)=\nabla^{\widehat g}_X\widehat \nu$, the identity holds.
\end{proof}
\printbibliography
\end{document}